\documentclass{amsart}

\usepackage[foot]{amsaddr}

\author{Philip Arathoon}
\author{Matthew D. Kvalheim}

\address
{Department of Mathematics, University of Michigan,
	Ann Arbor, MI, USA}

\email{philash@umich.edu, kvalheim@umich.edu}

\makeatletter
\@namedef{subjclassname@2020}{%
	\textup{2020} Mathematics Subject Classification}
\makeatother   
\title[]{Koopman embedding and super-linearization counterexamples with isolated equilibria}
\usepackage{import} 

\usepackage{enumitem} 

\usepackage[utf8]{inputenc}
\usepackage[T1]{fontenc}
\usepackage{lmodern}
\usepackage{amsmath}
\usepackage{amssymb}
\usepackage{amsthm}
\usepackage{stmaryrd} 
\usepackage{mathtools}
\usepackage{tikz-cd}
\usepackage{graphicx}
\usepackage{subfig}

\usepackage{pgfplots}
\pgfplotsset{width=10cm,compat=1.11}

\usepackage{thmtools}
\usepackage{thm-restate} 

\newcommand{\concept}[1]{\textbf{#1}}

\newcommand{\N}{\mathbb{N}}

\newcommand{\R}{\mathbb{R}}

\newcommand{\Sym}{\textnormal{Sym}}

\newcommand{\hooklongrightarrow}{\lhook\joinrel\longrightarrow}

\theoremstyle{definition}

\newtheorem{Th}{Theorem}

\newtheorem{Prop}{Proposition}

\newtheorem*{Quest-non}{Question}

\newtheorem*{Th-non}{Theorem}

\newcommand{\thistheoremname}{}
\newtheorem*{genericthm}{\thistheoremname}
{\renewcommand{\thistheoremname}{Theorem~\ref{#1}$'$}%
	\begin{genericthm}}
	{\end{genericthm}}

\newtheorem*{Def*}{Definition}
\newtheorem{Ex}{Example}

\newtheorem{Rem}{Remark}

\usepackage{marvosym}
\usepackage[
hypertexnames,%
citecolor=blue,%
colorlinks=true,%
linkcolor=red%
]{hyperref}
\usepackage{caption}

\begin{document}
	
	\begin{abstract}
    A frequently repeated claim in the ``applied Koopman operator theory''  literature is that a dynamical system with multiple isolated equilibria cannot be linearized in the sense of admitting a smooth embedding as an invariant submanifold of a linear dynamical system.
    This claim is sometimes made only for the class of super-linearizations, which additionally require that the embedding ``contain the state''.
    We show that both versions of this claim are false by constructing (super-)linearizable smooth dynamical systems on $\R^k$ having any countable (finite) number of isolated equilibria for each $k>1$.
	\end{abstract}
	\maketitle

    A linearizing embedding of a nonlinear smooth dynamical system is a global identification of the nonlinear system with an invariant submanifold of a linear dynamical system.
    Linearizing embeddings have been studied by various communities and are of central importance in the rapidly developing ``applied Koopman operator theory'' literature \cite{brunton2022modern}.
    An oft-repeated claim in that literature is that any dynamical system with multiple isolated equilibria cannot be linearized by a smooth embedding, or at least not by an embedding that ``contains the state''.
    We call the latter type of linearizing embeddings  super-linearizations.
    
    The first claim was shown to be false by the authors in \cite[Ex.~4]{kvalheim2023linearizability} if non-Euclidean state spaces are allowed.  
    In the present paper we show that both claims are also false for Euclidean state spaces by constructing for each $k > 1$ (i) linearizable dynamical systems on $\R^k$ having any countable number of isolated equilibria and (ii) super-linearizable dynamical systems on $\R^k$ having any finite number of equilibria.
    Thus, there are more (super-)linearizable dynamical systems than previously believed.

    Super-linearizations are of practical importance for engineering applications since they are invertible in closed form.  
    Our notion of super-linearization is slightly different from that of Belabbas and Chen \cite{belabbas2023sufficient} since we consider embeddings into linear rather than affine dynamical systems.

	We now proceed more formally.
	In this paper all manifolds and maps between them are smooth ($C^\infty$), and embeddings are smooth embeddings \cite{lee2013smooth}.
	Let \(M\) be a manifold and \(\Phi\colon~\R\times M\rightarrow M\) be the flow of a dynamical system. A map \(f\colon M_1\rightarrow M_2\) between two such dynamical systems \((M_1,\Phi_1)\) and \((M_2,\Phi_2)\) is called \concept{equivariant} (also called a semi-conjugacy) if 
	\begin{equation}
		\Phi_2^t\circ f=f\circ\Phi_1^t
	\end{equation}
	for all \(t\in \R\). If \(f\) is a diffeomorphism then we say that \((M_1,\Phi_1)\) and \((M_2,\Phi_2)\) are \concept{smoothly conjugate dynamical systems}.
	
	A dynamical system \((M,\Phi)\) admits a \concept{linearizing embedding} if there exists an equivariant embedding \(f\colon M\to \R^n\) of \((M,\Phi)\) into \((\R^n,\Psi)\) where \(\Psi^t=\text{exp}(At)\) is the flow of a linear system of ordinary differential equations on \(\R^n\) generated by some matrix \(A\). Recall that \(f\) is an embedding if it is a homeomorphism of \(M\) onto its image and if the derivative \(df\) is injective \cite[p.~85]{lee2013smooth}.
	
	From now on we shall assume that \(M\) is an open subset of \(\R^k\). For such systems there is a stronger notion of linearizability: a linearizing embedding \(f\colon M\rightarrow\R^{k+m}\) is a \concept{super-linearizing embedding} if it is of the form \(f(x)=(x,p(x))\) for some map \(p\colon M\rightarrow\R^m\). 
	Observe that the image of \(f\) is the graph of \(p\),
	\[
	\{(x,p(x))~|~x\in M\}.
	\]
	For this reason it will be helpful for us to refer to embeddings \(f\colon M\rightarrow\R^n\) as \concept{graphlike} if the image of \(f\) can be written in the form 
	\[
	\{x+\varphi(x)~|~x\in N\}
	\]
	where \(N\) is an open subset of some \(k\)-dimensional subspace of \(\R^n\) and \(\varphi\colon N\rightarrow U\) is a smooth map into a complementary subspace \(U\). 
	Equivalently, \(f\) is graphlike if and only if there exists a linear subspace \(U\subset\R^n\) of codimension \(k\) whose affine translates transversely intersect the image of \(f\) in at most one point.
	Every super-linearizing embedding is graphlike, and the image of  any graphlike linearizing embedding of a dynamical system is the image of a super-linearizing embedding of some smoothly conjugate dynamical system.

	\begin{Rem}
		Suppose \(f\colon M_1\rightarrow M_2\) defines a smooth conjugacy between two dynamical systems and \(g\) is a linearizing embedding of \((M_2,\Phi_2)\). Then \(g\circ f\) is a linearizing embedding of \((M_1,\Phi_1)\). Therefore, the property of admitting a linearizable embedding is well defined on the equivalence classes of smoothly conjugate dynamical systems. However, if \(g\) is a super-linearizing embedding then \(g\circ f\) need not be. Put differently, unlike linearizability, being super-linearizable is not a manifestly diffeomorphism-invariant attribute of dynamical systems.  
	\end{Rem}

\begin{Th}\label{claim}
	For any \(k>1\) there exists a super-linearizable dynamical system on \(\R^k\) with any given finite number of isolated equilibria.
\end{Th}
	
	\begin{Ex}[A linearizing embedding of a planar system with two isolated equilibria]\label{first_example}
		Consider the linear system on \(\R^3\) given by
	\begin{equation}\label{linflow}
		\frac{d}{dt}\begin{pmatrix}
			x \\ y \\ z
		\end{pmatrix}=
	\left(
	\begin{array}{cc|c}
		0 & 1 & 0\\
		1 & 0 & 0\\
		\hline 
		0 & 0 & 0
	\end{array}
	\right)\begin{pmatrix}
		x \\ y \\ z
	\end{pmatrix}.
	\end{equation}
	This preserves the planes \(z=\text{constant}\) and generates the standard flow on hyperbolae in \(xy\)-space. For each \(z=\text{constant}\) the lines \(y=\pm x\) divide the plane into four invariant quadrants. Let \(A_k\) denote the plane \(z=k\) with the quadrant containing \((x,y)=(-1,0)\) removed. Now let \(\gamma\) be a smooth curve in the plane \(y=0\) which connects \((x,z)=(0,1)\) to the origin as shown in Figure~\ref{twopoint} \cite[Ch.~2]{lee2013smooth}, and let \(B\) denote the set of all orbits of \eqref{linflow} intersecting \(\gamma\). Consider the surface
	\begin{equation}
		\Sigma=A_{1}\cup B\cup A_{0}
	\end{equation}
as shown in Figure~\ref{surface}. This surface is smooth and diffeomorphic to the plane, and hence, implicitly defines a linearizing embedding \(f\colon \R^2\hookrightarrow \R^3\) of a planar system with two isolated equilibria as shown in the figure.
\end{Ex}

\begin{figure}
	\includegraphics[scale=0.7]{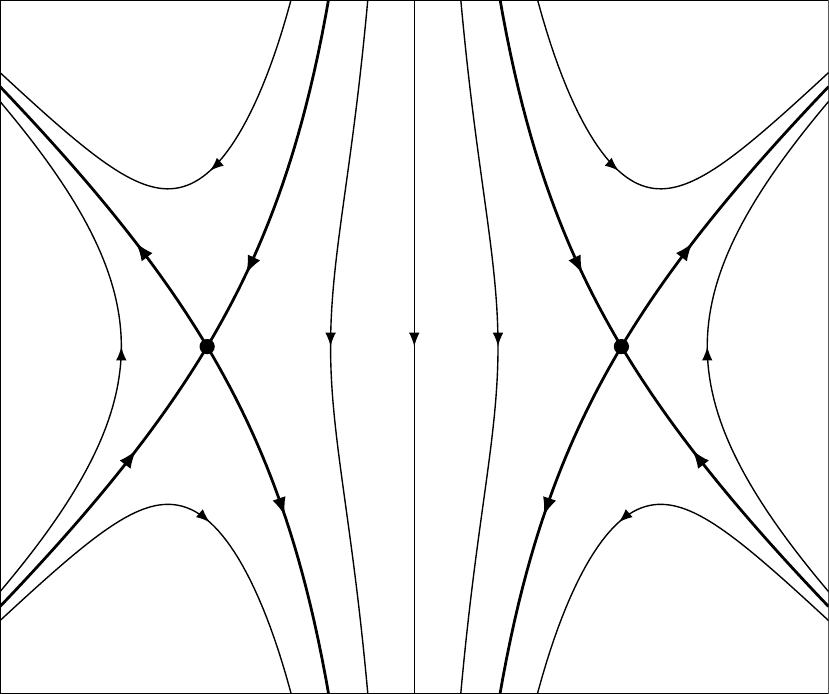}
	\vspace{2em}
	
	\includegraphics{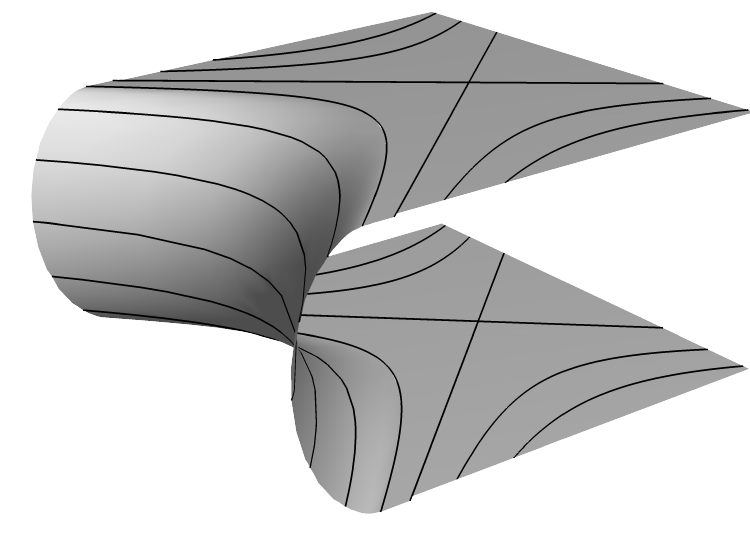}
	\caption{\label{surface}The phase portrait of a planar system with two isolated equilibria together with the image $\Sigma$ of a linearizing embedding into $\R^3$.}
\end{figure}

\begin{Rem}\label{extended}
	We may extend the previous example to give a planar system with any countable number $\ell\in \N\cup \{\infty\}$ of isolated equilibria. We do this by continuing to stack higher planes \(A_{2}, A_{3}, A_{4},\dots\) and alternately removing the quadrants containing \((x,y)=(\pm 1,0)\). The curve \(\gamma\) must now smoothly snake upwards joining the equilibria together, as shown for instance in Figure~\ref{fourpoint}. We shall denote the surface with exactly \(l> 2\) equilibria constructed in this way by \(\Sigma^l\).
\end{Rem}

\begin{Rem}\label{big_dim}
	We may enlarge such systems on the plane to \(\R^{k}\) for any \(k\ge 2\) by writing \((x,w)\in \R^2\times\R^{k-2}\) and setting, for instance, \(\dot{w}=w\). The map \((x,w)\mapsto (f(x),w)\) now defines a linearizing embedding of a system on \(\R^k\) with any desired countable number of isolated equilibria.
\end{Rem}

The surface \(\Sigma\) is not given by a graphlike embedding of \(\R^2\) into \(\R^3\). However, it might be possible to find an equivariant embedding of \(\R^3\) into some higher dimensional vector space whose restriction to  \(\Sigma\) is graphlike and hence a super-linearizing embedding of some dynamical system. This is the main idea that we use in our proof of Theorem~\ref{claim}. Before presenting the proof we must first discuss equivariant polynomial maps between vector spaces.

Consider a finite-dimensional vector space \(V\). The symmetric product \(\Sym^mV\) can be interpreted as homogeneous degree-\(m\) polynomials on the dual \(V^*\) by taking the evaluation of \(\eta\in V^*\) on \(v_1\odot\dots\odot v_m\) to be the product
\[
\langle\eta,v_1\rangle\cdots\langle\eta,v_m\rangle.
\]
Here \(\langle~,~\rangle\) denotes the pairing between \(V\) with its dual and \(\odot\) is the symmetrized tensor product \cite[pp.~473--474]{fulton1991representation}. We can then introduce the direct sum
\begin{equation}
	P^m(V^*)=\R\oplus V\oplus\Sym^2V\oplus\cdots\oplus\Sym^mV
\end{equation}
interpreted as the vector space of all degree-\(m\) polynomials on \(V^*\). The diagonal inclusion \(V\hookrightarrow\Sym^mV\) given by \(v\mapsto v\odot\cdots\odot v\) is a \(GL(V)\)-equivariant degree-\(m\) polynomial map, and by extension we may consider the natural map
\begin{equation}
	\Delta^m\colon V\hooklongrightarrow P^m(V^*).
\end{equation}
More explicitly, \(\Delta^m\) sends \(v\in V\) to the polynomial which when evaluated on \(\eta\in V^*\) yields
\[
\langle\eta,v\rangle+\langle\eta,v\rangle^2+\dots+\langle\eta,v\rangle^m.
\]
Note that \(\Delta^m\) is a smooth embedding of \(V\) into \(P^m(V^*)\) and is \(GL(V)\)-equivariant with respect to the action on polynomials \(p\) given by
\[
(g\cdot p)(\eta)=p(g^*\eta)
\]
for \(g\in GL(V)\) and where \(g^*\) is the adjoint. Thus, if \(f\) is a linearizing embedding, so is \(\Delta^m\circ f\).

\begin{Prop}\label{only_prop}
	Let \(f\) be a smooth embedding of \(\R^k\) into a real vector space \(V\). The composition \(\Delta^m\circ f\colon\R^k\hookrightarrow P^m(V^*)\) is a graphlike embedding if and only if there exist polynomials \(p_1,\dots, p_k\) of degree \(m\) on \(V\) for which the fibres of \(p=(p_1,\dots,p_k)\colon V\rightarrow\R^k\) intersect the image of \(f\) transversely in at most one point. 
\end{Prop}
We will say that the submanifold \(\text{Im}(f)\) is \concept{tamed} by the polynomials \(p_1,\dots, p_k\).
\begin{proof}
	Recall that an embedding \(\R^k\hookrightarrow W\) is graphlike if and only if there exists a codimension-\(k\) subspace of \(W\) whose affine translates intersect the image of the embedding transversely in at most one point. Equivalently, there exist linear functions \(\eta_1,\dots,\eta_k\in W^*\) for which each fibre of \(\eta=(\eta_1,\dots,\eta_k)\colon W\rightarrow\R^k\) intersects the image of the embedding transversely in at most one point. Consequently, \(\Delta^m\circ f\) is graphlike if and only if there exist \(\eta_1,\dots,\eta_k\) in \((P^m(V^*))^*\) with this property. The result follows by noting that linear functions on \(P^m(V^*)\) pull back through \(\Delta^m\) to polynomials of degree \(m\) on \(V\).
\end{proof}
\begin{Ex}[A super-linearizing embedding of a planar system with two isolated equilibria]\label{second_example}
	Consider the planar dynamical system with linearizing embedding \(f\colon\R^2\hookrightarrow\R^3\) from Example~\ref{first_example}. Using Proposition~\ref{only_prop} we will show that this admits a super-linearizing embedding into the much larger 20-dimensional vector space \(P^3({\R^3}^*)\) by finding two degree-3 polynomials \(q\) and \(p\) on \(\R^3\) which tame the surface \(\Sigma\). 
	
	We begin by choosing \(q(x,y,z)=y\). Now consider the intersections of \(\Sigma\) with the planes \(y=\text{constant}\), as shown for instance in Figure~\ref{twopoint} for \(y=0\). For any fixed value of \(y\), this intersection viewed in the \(xz\)-plane extends in the negative \(x\)-direction no further than \(x=-\sqrt{1+y^2}\), since \(x^2-y^2\) is constant along orbits. Therefore, the curve 
	\begin{equation}
		z=\frac{\kappa}{x+(1+y^2)}+\frac{1}{2}
	\end{equation}
for any \(\kappa\) and for \(y\) fixed, intersects \(\Sigma\cap\{y=\text{constant}\}\) transversely in at most one point. Hence, we set \(p(x,y,z)=(z-\frac{1}{2})\left(x+(1+y^2)\right)\) to establish that \(q\) and \(p\) tame \(\Sigma\) as intended.
\end{Ex}

\begin{figure}
	\includegraphics[scale=1]{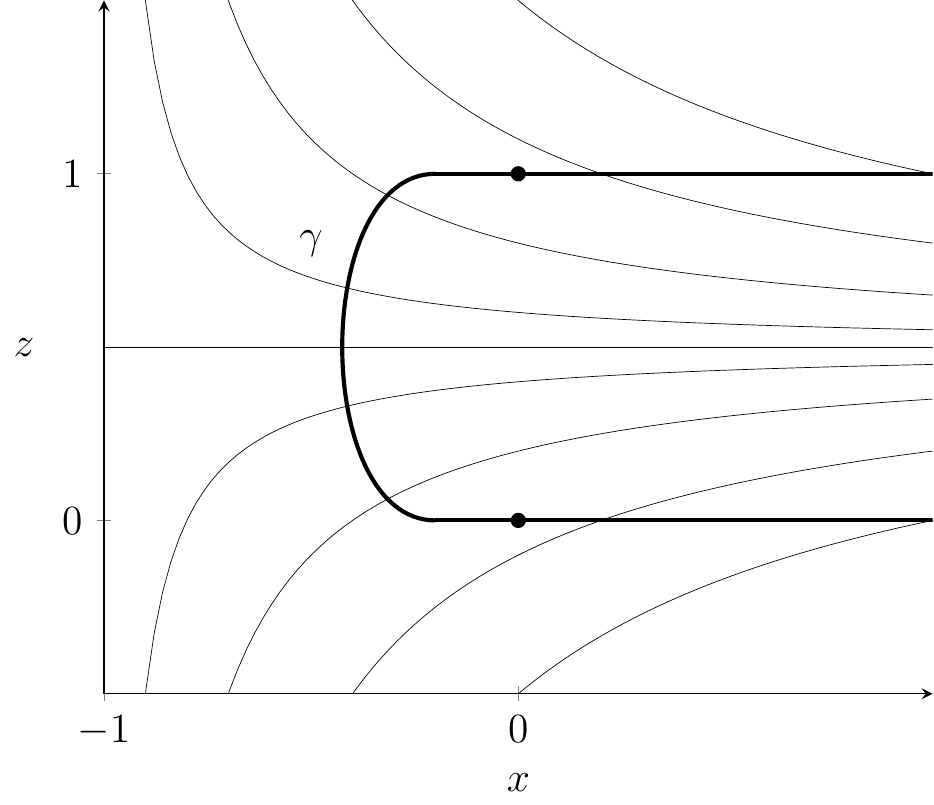}
	\caption{\label{twopoint}The intersection of the surface $\Sigma$ with $y=0$ is shown in bold. Also shown are contours of the polynomial $p(x,y,z)=(1+y^2)z+x(z-\frac{1}{2})$.}
\end{figure}

\begin{proof}[Proof of Theorem \ref{claim}]
	From the previous example we have already established the claim for the case of two isolated equilibria and \(k=2\). 
	By the technique of Remark~\ref{big_dim} it suffices to show that the extended surfaces \(\Sigma^l\) constructed in Remark~\ref{extended} for finite $\ell>2$ can also be tamed by some degree \(m\) polynomials \(q\) and \(p\). By Proposition~\ref{only_prop} this will imply that \(\Delta^m\circ f\) is an equivariant graphlike embedding, and hence, a super-linearizing embedding of some smoothly conjugate dynamical system.
	
	We again set \(q(x,y,z)=y\) and consider the intersections \(\Sigma^l\cap\{y=\text{constant}\}\) as shown for instance in Figure~\ref{fourpoint} for \(l=4\). We claim that for any fixed \(y\), the level sets of
	\begin{equation}
		p(x,y,z)=(1+y^{2})^{l-1}Mz+\textstyle x(z-\frac{1}{2})(z-\frac{3}{2})\cdots(z-l+\frac{3}{2})
	\end{equation}
transversely intersect the curve \(\gamma=\Sigma^l\cap\{y=\text{constant}\}\), where \(M\) is any positive constant larger than
\[
\max_{(x,z)\in R} \left\| x\frac{d}{dz}\textstyle (z-\frac{1}{2})(z-\frac{3}{2})\cdots(z-l+\frac{3}{2})\right\|
\]
and \(R\) is some closed rectangular region in the plane \(y=0\) with \(|x|<1\) and which contains all of the turns of \(\gamma\), as shown for instance in Figure~\ref{fourpoint}. To see why this is true consider the gradient of \(p\) in the plane \(y=\text{constant}\), 
\begin{align*}
	p_x&=\textstyle (z-\frac{1}{2})(z-\frac{3}{2})\cdots(z-l+\frac{3}{2})\\
	p_z&=(1+y^2)^{l-1}M+x\frac{d}{dz}\textstyle (z-\frac{1}{2})(z-\frac{3}{2})\cdots(z-l+\frac{3}{2}).
\end{align*}
For different fixed \(y\), rectangular regions containing the orbits through the turns of \(\gamma\) can be chosen to scale in the \(x\)-direction by a factor less than \(\sqrt{1+y^2}\). Therefore, by construction \(p_z>0\) along the turns of \(\gamma\). Furthermore, the sign of \(p_x\) always agrees with the sign of the \(x\)-derivative of \(\gamma\). It follows that the derivative of \(p\) along the curve \(\gamma\) as it moves upwards is positive everywhere. Each joint level set of \(q\) and \(p\) therefore intersects the surface \(\Sigma^l\) transversely exactly once, as desired.
\end{proof}

\begin{figure}
	\includegraphics[scale=1]{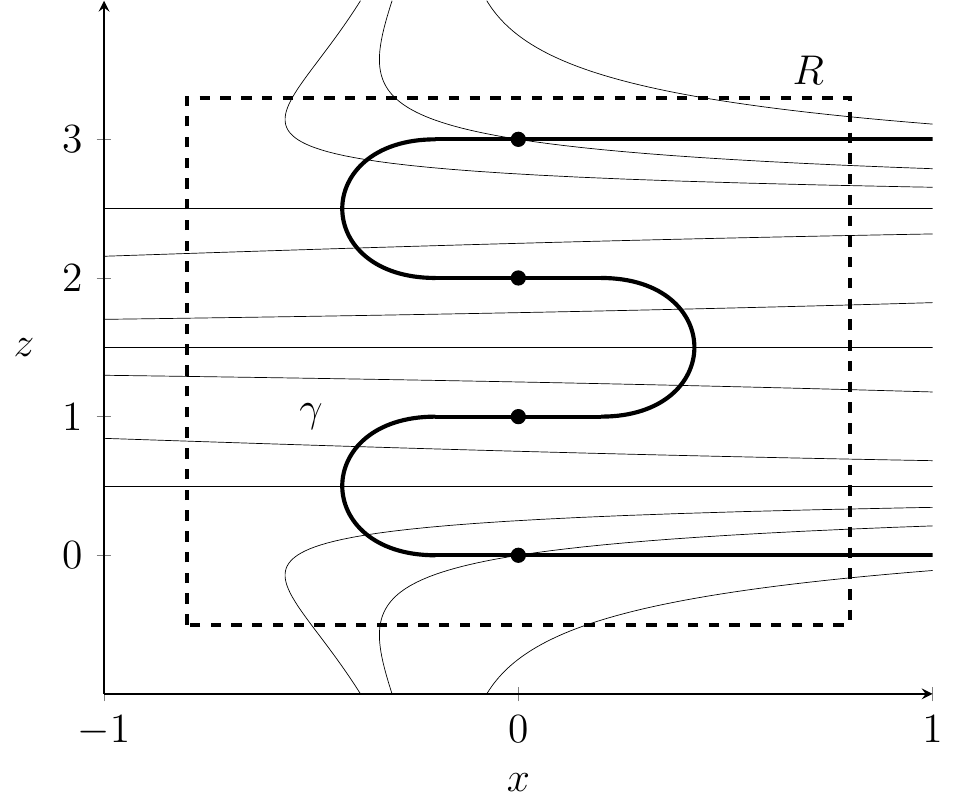}
	\caption{\label{fourpoint}The intersection of the surface $\Sigma^4$ with $y=0$ is shown in bold. Also shown are contours of the polynomial $p(x,y,z)=4z(1+y^2)^3+x(z-\frac{1}{2})(z-\frac{3}{2})(z-\frac{5}{2})$.}
\end{figure}

\begin{Rem}
	It is tempting to extend this proof to include the surface in Example~\ref{extended} with countably infinite isolated equilibria. However, a direct generalisation will not work. To see why, consider again the intersection \(\gamma=\Sigma^\infty\cap\{y=0\}\). This intersection is now a curve which snakes upwards with infinitely many turns. Suppose this curve can be tamed by some polynomial \(p(x,z)\). Then the derivative of \(p\) along \(\gamma\) must be nowhere zero, and hence \(p_x\) must alternate sign infinitely often along the line \(x=0\). This contradicts the Fundamental Theorem of Algebra since \(p_x(0,z)\) is a polynomial in \(z\).
\end{Rem}

Incidentally, the example of \(\gamma\subset\R^2\) in the previous remark shows that not every embedded submanifold can be tamed by polynomials. If one is interested in applying a polynomial embedding \(\Delta^m\) to super-linearize a dynamical system, it would be of interest to obtain necessary and sufficient conditions for an embedded submanifold to be tamed by polynomials, and to perhaps bound their degree.

To conclude, we would like to make some comments regarding the relevance of these counterexamples to the study of Koopman eigenmappings. For a dynamical system \((M,\Phi)\) a Koopman eigenfunction is a function \(\psi\) on \(M\) for which
\[
\frac{d}{dt}\psi\circ\Phi^t=\lambda\psi
\]
for all \(t\) and for some \(\lambda\in\R\). If one has multiple Koopman eigenfunctions \(\psi_1,\dots,\psi_n\) then \(x\mapsto(\psi_1(x),\dots,\psi_n(x))\) defines an equivariant map from \(M\) into a linear dynamical system on \(\R^n\). If one has enough functionally independent Koopman eigenfunctions, then this map will be a smooth immersion, and hence, is very closely related to our notion of a linearizing embedding. Indeed, in Example~\ref{first_example} the three functions \(x+y\), \(x-y\), and \(z\) each pull back through the embedding to give Koopman eigenfunctions on a planar dynamical system with multiple isolated equilibria. 

As far as we know, these are the first examples of {globally defined} smooth Koopman eigenfunctions for a dynamical system with multiple isolated equilibria. Furthermore, the linear span of these eigenfunctions seems to be the first known example of a non-trivial finite-dimensional subspace of smooth functions for a dynamical system with multiple isolated equilibria which is invariant under the action of the Koopman operator. Moreover, by using a super-linearizing embedding (as in Example~\ref{second_example}) such an invariant subspace additionally ``contains the state''.

Finally, we should caution that despite the counterexamples we have presented, there exist many obstructions which prevent the existence of linearizing embeddings for general dynamical systems. The existence of hetero- and homoclinic orbits between equilibria is one such obstruction (since these orbit types are not possible for linear systems),  and so too is the presence of multiple asymptotically stable equilibria (see the first footnote in \cite{kvalheim2023linearizability}). In addition, the main theorem of \cite{liu2023non} forbids the existence of linearizing embeddings for systems whose collection of \(\omega\)-limit set is countable and contains more than one element, and whose forward-orbits are all precompact. We note that this is consistent with our example since it contains unbounded forward-orbits which are not precompact.

%


\providecommand{\bysame}{\leavevmode\hbox to3em{\hrulefill}\thinspace}
\providecommand{\MR}{\relax\ifhmode\unskip\space\fi MR }
\providecommand{\MRhref}[2]{%
	\href{http://www.ams.org/mathscinet-getitem?mr=#1}{#2}
}
\providecommand{\href}[2]{#2}

\end{document}